\documentclass[a4paper,11pt]{article}
\usepackage[T1]{fontenc}
\usepackage[bitstream-charter]{mathdesign}
\usepackage[french,english]{babel}
\usepackage[babel=true,kerning=true]{microtype}
\usepackage[utf8]{inputenc} 
\usepackage{graphicx}
\usepackage{amsmath}
\usepackage{amsthm}
\usepackage{geometry}
\usepackage{tikz}
\usepackage{hyperref}

\hypersetup{
colorlinks=true,
pdfauthor={Mickaël Kourganoff},
pdftitle={},
pdfcreator=PDFLaTeX,
pdfproducer=PDFLaTeX,
linkcolor=[rgb]{0,0,0.7},
citecolor=[rgb]{0.7,0,0},
}
\usetikzlibrary{shapes.geometric}
\usetikzlibrary{through,calc}
\usetikzlibrary{plotmarks}
\usetikzlibrary{arrows,decorations.markings}

\tikzset{
    scale plot marks/.is choice,
    scale plot marks/false/.code={
        \def\pgfuseplotmark##1{\pgftransformresetnontranslations\csname pgf@plot@mark@##1\endcsname}
    },
    scale plot marks/true/.style={},
    scale plot marks/.default=true
}

\tikzset{
    partial ellipse/.style args={#1:#2:#3}{
        insert path={+ (#1:#3) arc (#1:#2:#3)}
    }
}

\tikzset{fleche/.style args={#1:#2}{ postaction = decorate,decoration={name=markings,mark=at position #1 with {\arrow[#2,scale=2]{>}}}}}

\newtheorem{prop}{Proposition}[section]
\newtheorem{thm}[prop]{Theorem}
\newtheorem{lemme}[prop]{Lemma}

\newtheorem{cor}[prop]{Corollary}

\theoremstyle{remark}

\theoremstyle{definition}

\newtheorem{definition}[prop]{Definition}
\newtheorem{exemple}[prop]{Example}

\newcommand{\abs}[1]{\left\lvert #1 \right\rvert}
\newcommand{\Int}{\mathrm{Int}~}
\newcommand{\norm}[1]{\left\lVert #1 \right\rVert}
\newcommand{\setof}[2]{\left\{ #1 ~ \middle\arrowvert ~ #2 \right\}}
\newcommand{\prodscal}[2]{\left\langle #1 ~ \middle\arrowvert ~ #2 \right\rangle}

\newcommand{\II}{\mathrm{II}}

\begin{document}
\title{Embedded surfaces with Anosov geodesic flows, approximating spherical billiards}
\author{Mickaël Kourganoff\footnote{Institut Fourier, Grenoble University, France}}
\date{}

\maketitle

\begin{abstract}
We consider a billiard in the sphere $\mathbb S^2$ with circular obstacles, and give a sufficient condition for its flow to be uniformly hyperbolic. We show that the billiard flow in this case is approximated by an Anosov geodesic flow on a surface in the ambiant space $\mathbb S^3$. As an application, we show that every orientable surface of genus at least $11$ admits an isometric embedding into $\mathbb S^3$ (equipped with the standard metric) such that its geodesic flow is Anosov. Finally, we explain why this construction cannot provide examples of isometric embeddings of surfaces in the Euclidean $\mathbb R^3$ with Anosov geodesic flows.
\end{abstract}

\section{Introduction}

\subsection{Anosov geodesic flows for embedded surfaces}
The geodesic flow of any Riemannian surface whose curvature is negative everywhere is Anosov: in particular, any orientable surface of genus at least $2$ can be endowed with a hyperbolic metric, for which the geodesic flow is Anosov. On the contrary, there is no Riemannian metric on the torus or the sphere with an Anosov geodesic flow.

If a closed surface admits an isometric embedding in $\mathbb R^3$, then it needs to have positive curvature somewhere. However, it is still possible to obtain an Anosov geodesic flow for such a surface, as shown by Donnay and Pugh~\cite{donnay2003anosov}. More precisely, they showed that there exists a genus $g_0$ such that for all $g \geq g_0$, the orientable surface of genus $g$ admits an isometric embedding in $\mathbb R^3$ whose geodesic flow is Anosov (see~\cite{MR2041262}). The value of $g_0$ is completely unknown: in particular, it is unknown whether it is possible to embed isometrically a surface of genus smaller than one million in $\mathbb R^3$ so that its geodesic flow is Anosov.

The same question may be asked for embeddings in the sphere $\mathbb S^3$ endowed with the standard metric. Here, we have the following situation:
\begin{prop}
Any closed surface $M$ isometrically embedded in $\mathbb S^3$ admits at least one point at which the Gauss curvature is at least $1$ (except if $M$ is a torus or a Klein bottle).
\end{prop}
\begin{proof}
The curvature of the surface is given at each point by $K = k_1 k_2 + 1$, where $k_1$ and $k_2$ are the principal curvatures of the surface. If $K < 1$ everywhere, then the principal curvatures are nonzero and have different signs at each point. Thus the principal directions induce two nonsingular vector fields on $M$, which implies that the Euler characteristic of $M$ must be zero.
\end{proof}

We will show for the first time that it is possible to embed isometrically in $\mathbb S^3$ a Riemannian surface with Anosov geodesic flow. More precisely:
\begin{thm} \label{thmGenus11}
Every orientable surface of genus at least $11$ admits an isometric embedding into $\mathbb S^3$ such that its geodesic flow is Anosov.
\end{thm}

\subsection{Approximating billiards by flattened surfaces}

Birkhoff~\cite{birkhoff1927dynamical} seems to be the first to suggest that billiards could be approximated by geodesic flows on flattened surfaces: he took the example of an ellipsoid whose vertical axis tends to zero, which converges to the billiard in an ellipse. Later, Arnold~\cite{arnold1963small} writes that smooth Sinaï billiards could be approximated by surfaces with nonpositive curvature, with Anosov geodesic flows. In~\cite{MR3508162}, we proved this fact for a large class of flattened surfaces.

In this paper, we prove a similar result for another class of objects. Consider a finite family of open disks $\Delta_i$, whose closures are disjoint, which have radii $r_i < \pi/2$, on the sphere $\mathbb S^2$: we will say that the billiard $D = \mathbb S^2 \setminus \bigcup_{i} \Delta_i$ is a \emph{spherical billiard with circular obstacles} (Figure~\ref{figBilliard}). Define the billiard flow in the following way: outside the obstacles, the particle follows the geodesics of the sphere with unit speed; when the particle hits an obstacle, it bounces, following the usual billiard reflection law.

\begin{figure}[h!]
\centering
\includegraphics[width=200pt]{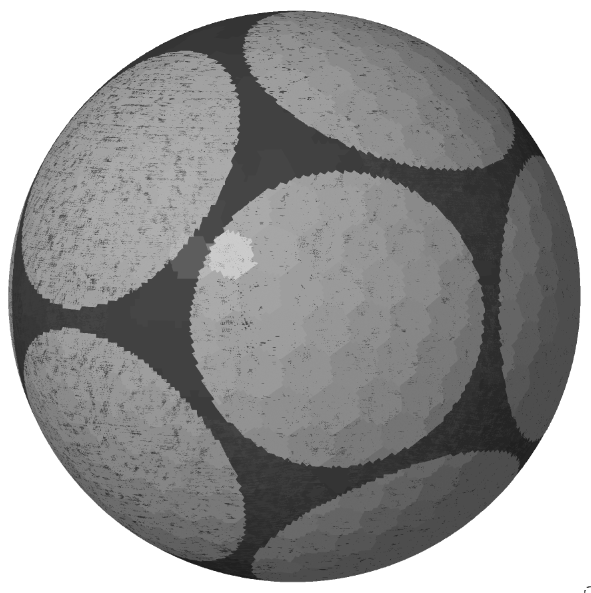}
\caption{A spherical billiard with $12$ obstacles.} \label{figBilliard}
\end{figure}

The \emph{horizon} $H$ of a spherical billiard is the length of the longest geodesic of $\mathbb S^2$ contained in the billard $D$.

The phase space $\Omega$ of the billiard is defined as $\Omega = T^1(\Int(D))$ (the unit tangent bundle of the interior of $D$). To define uniform hyperbolicity, we need to consider the set $\tilde \Omega$ of all $(x, v) \in T^1(\mathrm{Int}(D))$ such that the trajectory starting at $(x, v)$ does not contain any grazing collision (that is, each obstacle is reached with a nonzero angle).

We will use a definition of ``uniformly hyperbolic billiard'' which can be found in~\cite{MR2229799}:
\begin{definition} \label{defHyperbolicBilliard}
The billiard flow $\phi^t$ is \emph{uniformly hyperbolic} if at each point $x \in \tilde \Omega$, there exists a decomposition of $T_x\Omega$, stable under the flow,
\[ T_{x} \Omega = E_x^0 \oplus E_x^u \oplus E_x^s \]
where $E_x^0 = \mathbb R \left. \frac{d}{dt}\right|_{t=0} \phi^t(x)$, such that
\[ \lVert D\phi_x^t|_{E_x^s} \rVert \leq a \lambda^t, \  \lVert D\phi_x^{-t}|_{E_x^u} \rVert \leq a \lambda^{t} \]
(for some $a > 0$ and $\lambda \in (0,1)$, which do not depend on $x$).
\end{definition}	

Smooth flat billiards with negative curvature and finite horizon are known to be uniformly hyperbolic~\cite{MR0274721}, but it is not the case for spherical billiards, as the following example shows:

\begin{exemple}
Consider six disjoint disks on the sphere, with the same radius $r$, whose centers are the vertices of a regular octahedron which is inscribed in the sphere. If the radius $r$ is large enough, the billiard has finite horizon. However, there is a family of billiard trajectories which are parallel to the geodesic which is drawn on Figure~\ref{figOctahedron}, and thus, the billiard is not uniformly hyperbolic.
\end{exemple}

\begin{figure}
\centering
\includegraphics[width=200pt]{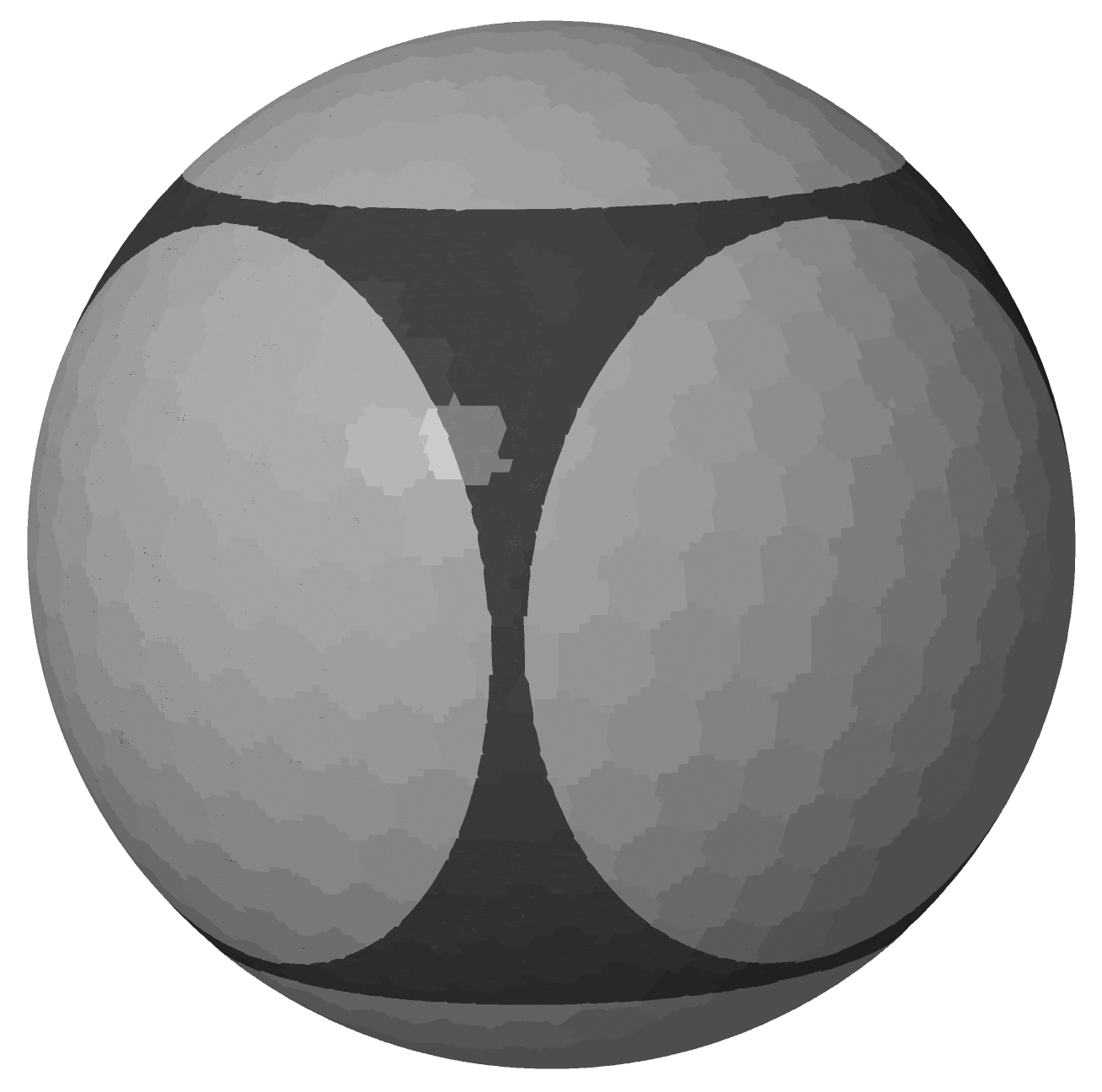}
\includegraphics[width=200pt]{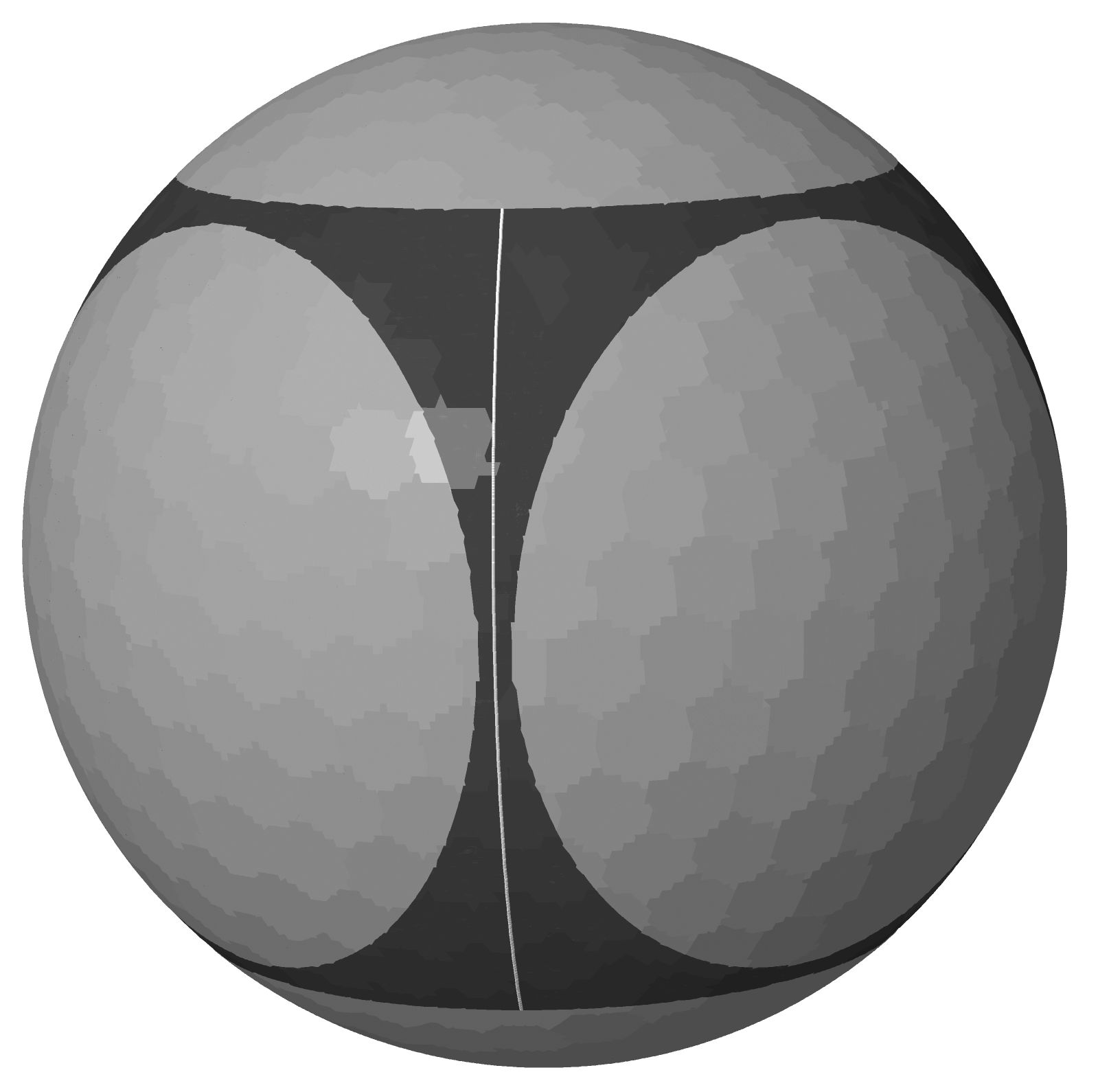}
\caption{The trajectory which is shown on the right-hand side is not hyperbolic.} \label{figOctahedron}
\end{figure}

However, we will show that there exists a spherical billiard with $12$ circular obstacles which is uniformly hyperbolic (see Figure~\ref{figBilliard}).

\subsection{Approximation by a closed surface}

In~\cite{MR3508162}, we have shown that, under some conditions, uniformly hyperbolic billiards may be approximated by smooth surfaces whose geodesic flow is Anosov. The main result of~\cite{MR3508162} only applies to flat billiards, but we will show in this paper that it is possible to approximate our spherical billiard by a surface in the ambiant space $\mathbb S^3$ such that the geodesic flow is Anosov (Theorem~\ref{thmAnosov}): see Figure~\ref{figSurface}.

\begin{figure}[h!]
\centering
\includegraphics[width=200pt]{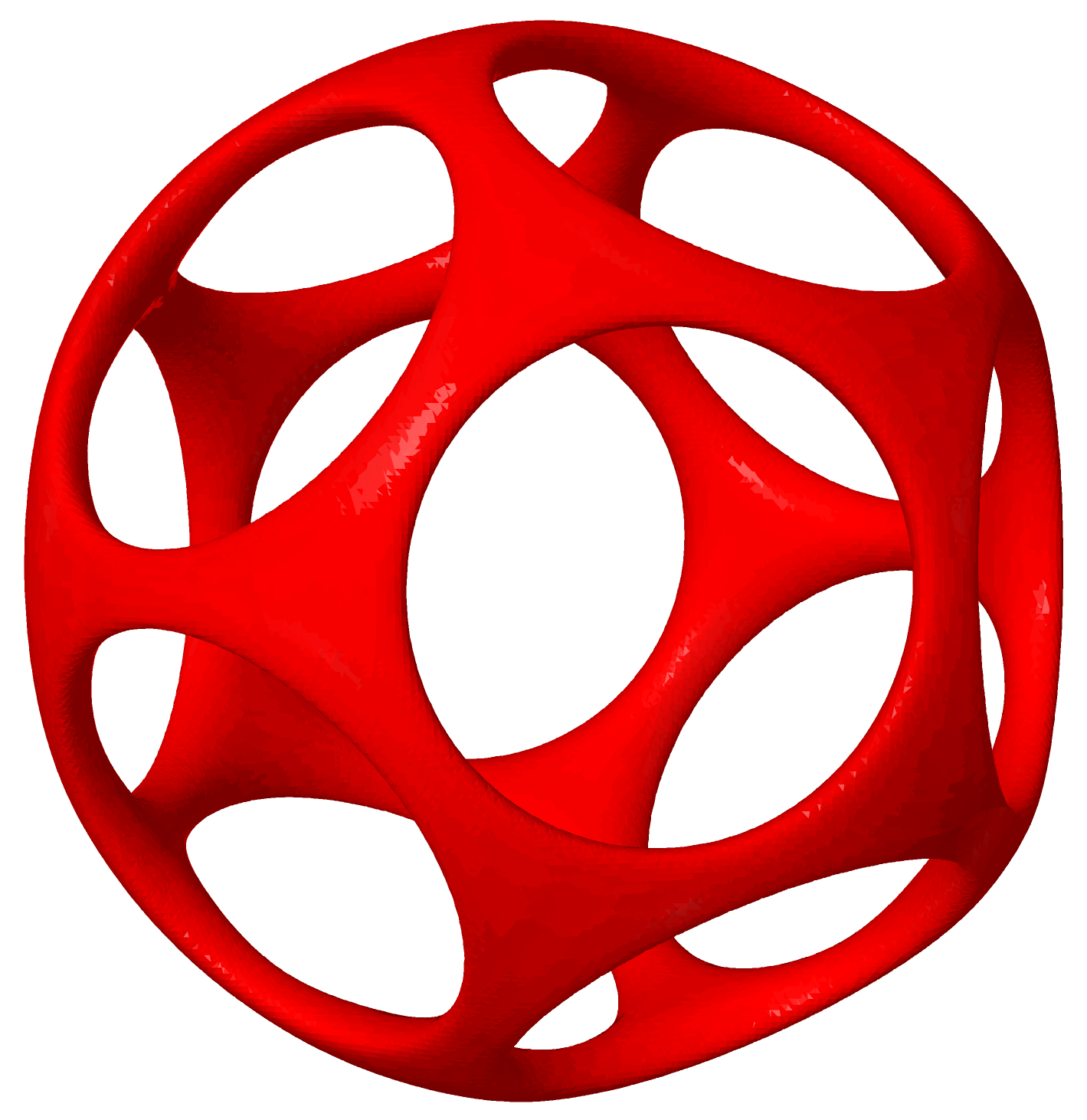}
\caption{An approximation of the spherical billiard by a surface of genus $11$. This surface is isometrically embedded in $\mathbb S^3$ and is seen here in stereographic projection. Its geodesic flow is Anosov.} \label{figSurface}
\end{figure}

It is tempting to try the same construction in the ambiant space is $\mathbb R^3$: however, we will see that in this framework, the geodesic flow of a surface which approximates a spherical billiard is \emph{never} Anosov (see Theorem~\ref{thmImpossible}). This result, which might seem surprising at first sight, is due to the accumulation of a high quantity of \emph{positive curvature} beside the negative curvature which appears near the boundary of the billiard.

\section{Main results}
%
Consider a billiard $D$ in the sphere $\mathbb S^2$ with circular obstacles. Consider the largest obstacle and the length of its radius $A$ (for the spherical metric), and the horizon $H$ of the billiard. First we will prove:

\begin{thm} \label{thmHypBilliard}
If $H < \pi/2$ and $2 \tan (\pi/2 - A) > \tan (H)$, then $D$ is uniformly hyperbolic. In particular, it is the case if $A + H < \pi/2$.
\end{thm}

We will now see how uniform hyperbolicity ``transfers'' to surfaces which approximate such a billiard.

%

Consider the stereographic projection of $\mathbb S^3$, and the surface \[ \mathbb S^2 = \setof{(x, y, z) \in \mathbb S^3}{x^2 + y^2 + z^2 = 1}. \] For $q \in \mathbb R^3$ denote by $\rho(q)$ the radial unit vector at $q$ (the unit vector which is positively colinear to the vector joining the origin to $q$) and by $\pi$ the natural projection of $\mathbb S^3$ (minus the poles) onto $\mathbb S^2$:
\[ \begin{aligned} \pi: \mathbb R^3 \setminus \{0\} & \to \mathbb S^2 \\ (x, y, z) & \to (x / (x^2 + y^2 + z^2), y / (x^2 + y^2 + z^2), z / (x^2 + y^2 + z^2)). \end{aligned} \]

We also consider the ``flattening map'' $f_\epsilon$ for $\epsilon \in (0,1)$:
\[ \begin{aligned} f_\epsilon: \mathbb R^3 \setminus \{0\} & \to \mathbb R^3 \setminus \{0\} \\ q & \to \epsilon q + (1-\epsilon) \pi(q). \end{aligned}. \]

The main theorem of this paper is the following:

\begin{thm} \label{thmAnosov}
Consider a spherical billiard $D$ with spherical obstacles $\Delta_i$, which satisfies $H < \pi/2$ and $2 \tan (\pi/2 - A) > \tan (H)$ (see the notations above), and a surface $\Sigma$ in $\mathbb S^3$ such that $\pi(\Sigma) = D$. We assume that:
\begin{enumerate}
\item (Transversality to the fibers of the projection.) For all $x \in \Sigma$, if $\pi(x) \not\in \partial D$, then $\rho \not\in T_x\Sigma$.
\item (Nonzero vertical principal curvature.) For all $x \in \Sigma$, if $\pi(x) \in \partial D$, then $\II_x(\rho) \neq 0$ (where $\II$ is the second fundamental form).
\item (Symmetry.) For all $i$, $\partial \Delta_i \subseteq \Sigma$; moreover, there is a neighborhood of $\partial \Delta_i$ in $\Sigma$ which is invariant by inversion with respect to $\mathbb S^2$, and by rotation in $\mathbb S^3$ around the axis $(0, c_i)$, where $c_i$ is the center of $\Delta_i$.
\end{enumerate}

Then there exists $\epsilon_0 \in (0,1)$ such that for all $\epsilon \leq \epsilon_0$, the geodesic flow on the flattened surface $\Sigma_\epsilon = f_\epsilon(\Sigma)$ is Anosov.
\end{thm}

The assumptions of Theorem~\ref{thmAnosov} are very close to those which appear in the main theorem of~\cite{MR3508162} (which deals with the case of flat billiards), but the proof is made more difficult by the positive curvature of the sphere.

It is actually simple, from a given billiard, to construct a surface which satisfies these assumptions. More precisely:
\begin{thm} \label{thmConstruction}
If $D$ is a spherical billiard with $n$ circular obstacles, then there exists a surface $\Sigma$ of genus $n-1$ such that $\pi(\Sigma) = D$, which satisfies the assumptions 1, 2 and 3 of Theorem~\ref{thmAnosov}.
\end{thm}
\begin{proof}
We assume that the circular obstacles have radii $r_1, r_2, \ldots, r_n$.
Choose $\delta > 0$ such that the circles $\tilde \Delta_i$ of radii $r_i + \delta$, with the same centers as the obstacles, remain disjoint. Consider the image $S_1$ of $\mathbb S^2 \setminus \bigcup_i \tilde \Delta_i$ by a homothety (in $\mathbb R^3$) of center $(0,0,0)$ and ratio $1-\epsilon$, with a small $\epsilon > 0$. Consider the image $S_2$ of $S_1$ by inversion with respect to $\mathbb S^2$. Finally, construct symmetric tubes which connect the pairs of ``holes'' on the surfaces $S_1$ and $S_2$.
\end{proof}

Theorem~\ref{thmConstruction} will allow us to prove Theorem~\ref{thmGenus11} in Section~\ref{sectGenus11}. On the other hand, we prove the following:

\begin{thm} \label{thmImpossible}
Consider a spherical billiard $D$ with spherical obstacles $\Delta_i$, of radii $r_i < \pi/2$, and a surface $\Sigma$ in $\mathbb R^3$ such that $\pi(\Sigma) = D$ and $\partial \Delta_i \subseteq \Sigma$. Write $r_{\max} = \max_i (r_i)$ and, for all $\delta > 0$, denote by $V_i^\delta$ the open neighborhood of $\partial \Delta_i$ in $\Sigma$ which consists of all points at distance less than $\delta$ from $\partial \Delta_i$.
We will say that the surface $\Sigma$ is $\epsilon$-$C^1$-close to the billiard $D$ if there exists locally a parametrization $f_1$ of $\Sigma$ and a parametrization $f_2$ of $D$ such that $\norm{f_1 - f_2}_{C^2} \leq \epsilon$ for the $C^2$-norm in the Euclidean $\mathbb R^3$.

For all $\delta_1 \in (0, \pi/2 - r_{\max})$, there exists $\delta_2 > 0$, such that the geodesic flow of any surface satisfying the following conditions has conjugate points:
\begin{enumerate}
\item (Symmetry.) The neighborhood $V_i^{\delta_1}$ is invariant by rotation in $\mathbb R^3$ around the axis $(0, c_i)$, where $c_i$ is the center of $\Delta_i$.
\item (Approximation of the billiard.) The surface $\Sigma \setminus \bigcup_i V_i^{\delta_2}$ is $\delta_2$-$C^2$-close to $D$.
\end{enumerate}
\end{thm}

In particular, for the surface $\Sigma_\epsilon$ in Theorem~\ref{thmAnosov}, endowed with the metric induced by $(\mathbb R^3, g_\mathrm{eucl})$, the geodesic flow always has conjugate points, and thus, it is never Anosov.

\paragraph{Structure of the paper.} The proof of Theorem~\ref{thmHypBilliard} relies on a theorem which is proved in~\cite{kourganoff2016anosov}: in Section~\ref{sectRiccati}, we recall the statement of this theorem and finish the proof of Theorem~\ref{thmHypBilliard}. In Section~\ref{sectFlatteningEuclidean}, we state and prove a lemma in the Euclidean ambiant space. This lemma is transposed to the spherical ambiant space in Section~\ref{sectCurvatureInTube}. The local dynamics near a circular obstacle are studied in Section~\ref{sectDynamicsInTube}. Section~\ref{sectEndProof} ends the proof of Theorem~\ref{thmAnosov} by studying the global dynamics. We prove Theorem~\ref{thmGenus11} in Section~\ref{sectGenus11}. Finally, we show Theorem~\ref{thmImpossible} in Section~\ref{sectImpossible}.

\section{The Riccati equation} \label{sectRiccati}

The Riccati equation is an important tool for the proofs of this paper. On a smooth Riemannian manifold $M$, the Riccati equation is a differential equation along a geodesic $\gamma: [a, b] \to \mathbb M$ given by:
\[ \dot u(t) = -K(t) - u(t)^2 \]
where $K$ is the Gaussian curvature of the surface.

The space of orthogonal Jacobi fields on the geodesic $\gamma$ has dimension $1$, thus it is possible to consider any orthogonal Jacobi field as a function $j : [a, b] \to \mathbb R$ (by choosing an orientation of the normal bundle of $\gamma$). In this case, it is well-known that $j$ satisfies the equation $j''(t) = - K(t) j(t)$. A short calculation then shows that $u = j' / j$ satisfies the Riccati equation whenever $j \neq 0$.

The following criterion is an improvement of a statement which appears in~\cite{donnay2003anosov}; a complete proof may be found in~\cite{kourganoff2016anosov}.
\begin{thm} \label{conditionConeRiccati}
Let $M$ be a closed surface. Assume that there exist $m > 0$ and $C > c > 0$ such that for any geodesic $\gamma: \mathbb R \to M$, there exists an increasing sequence of times $(t_k)_{k \in \mathbb Z} \in \mathbb R^{\mathbb Z}$ with $c \leq t_{k+1} - t_k \leq C$, such that the solution $u$ of the Riccati equation with initial condition $u(t_k) = 0$ is defined on the interval $[t_k, t_{k+1}]$, and $u(t_{k+1}) > m$. Then the geodesic flow $\phi_t : T^1 M \to T^1 M$ is Anosov.
\end{thm}

Now, consider a spherical billiard $D$ and a billiard trajectory $\gamma$. It is possible to consider a small variation of $\gamma$, which is also called a Jacobi field, and consider $u = j'/j$, as in the case of a geodesic flow. Thus, there is a natural generalization of the Riccati equation. We say that $u$ is a solution of this equation if:
\begin{enumerate}
\item in the interval between two collisions, $\dot u(t) = -K(t) - u(t)^2$ ;
\item when the particle bounces against the boundary at a time $t$, $u$ undergoes a discontinuity: we have $u(t^+) = u(t^-) - \frac{2 \kappa}{\sin \theta}$, where $\kappa$ is the geodesic curvature of the boundary of $D$, and $\theta$ is the angle of incidence\footnote{The notation $u(t^+)$ stands for $\lim_{h \to 0, h > 0} u(t + h)$, and likewise $u(t^-) = \lim_{h \to 0, h < 0} u(t + h)$.}.
\end{enumerate}

With this generalized Riccati equation, the following theorem holds (its proof may be found in~\cite{kourganoff2016anosov}):

\begin{thm} \label{thmRiccatiBilliard}
Consider a spherical billiard $D$. Assume that there exist positive constants $A, m, c$ and $C$ such that for any trajectory $\gamma$ with $\gamma(0) \in \tilde\Omega$, there exists an increasing sequence of times $(t_k)_{k \in \mathbb Z} \in \mathbb R^{\mathbb Z}$ satisfying $c \leq t_{k+1} - t_k \leq C$, such that for any $k \in \mathbb Z$, the solution $u$ of the Riccati equation with initial condition $u(t_k^+) = 0$ satisfies $u(t^+) \geq -A$ for all $t \in [t_k, t_{k+1}]$, and $u(t_{k+1}^+) > m$. Also assume that for each $k \in \mathbb Z$, there is no collision in the interval $(t_k - c, t_k)$, and at most one collision in the interval $(t_k, t_{k+1}]$. Then the billiard flow on $D$ is uniformly hyperbolic (see Definition~\ref{defHyperbolicBilliard}).
\end{thm}

Knowing this theorem, we are ready to prove Theorem~\ref{thmHypBilliard}.

\begin{proof}[Proof of Theorem~\ref{thmHypBilliard}]
Since the obstacles are circles of radius at most $A$, their geodesic curvature is at least $\tan (\pi/2 - A)$. We define $m = 2 \tan (\pi/2 - A) - \tan (H)$ (thus $m > 0$).

Consider a billiard trajectory $(q(t), p(t))_{t \in [0, 2H]}$ in the billiard $D$, with collision times $(t_k)_{k \in \mathbb Z}$. Fix $k$ and consider the solution $u$ of the Riccati equation along this trajectory with $u(t_k) = 0$. We want to show that $u(t_{k+1}) > \delta$.

On $(t_k, t_{k+1})$, the solution $u$ satisfies $u'(t) = -1 - u(t)^2$, so $u(t_{k+1}^-) = -\tan(t_{k+1} - t_k) \geq -\tan(H)$.

Knowing that $u(t_{k+1}^+) = u(t_{k+1}^-) - \frac{2 \kappa}{\sin \theta}$, we obtain $u(t_{k+1}^+) \geq -\tan(H) + 2 \tan (\pi/2 - A) = m$. According to Theorem~\ref{thmRiccatiBilliard}, this concludes the proof.
\end{proof}

\section{Flattening a curve in the Euclidean plane} \label{sectFlatteningEuclidean}

\begin{lemme} \label{lemmaHighCurvatureFlat}
Consider a smooth curve $c : (-a, a) \to \mathbb R^2$ with unit speed, and write its coordinates $c(t) = (r(t), z(t))$. Assume that $0 \in (a, b)$, $c(0) = R$, $c'(0) = (0,1)$, and the curvature of $c$ at $0$ is nonzero. For all $t \in (a, b)$, also assume that $c(t) \geq R$,
$r(-t) = r(t)$ and $z(-t) = -z(t)$. Consider the flattened curve $c^\epsilon(t) = (r(t), \epsilon z(t))$ and its curvature $k^\epsilon(t)$. The unit tangent vector to $c^\epsilon(t)$ is $T^\epsilon(t) = (T^\epsilon_r(t), T^\epsilon_z(t))$, and the normal vector is $(T_z^\epsilon(t), -T_r^\epsilon(t))$.

Then there exists $m_0 > 0$ such that for all $m \leq m_0$, there exists $\epsilon_0 > 0$ such that for all $\epsilon \leq \epsilon_0$, there exists $t_c$ such that
\begin{enumerate}
\item for all $t \in (0, t_c)$, $T^\epsilon_z(t) \geq m$ and $k^\epsilon(t) \geq m^4 / \epsilon^2$,
\item for all $t \in (t_c, m)$, $T^\epsilon_z(t) \leq m$ and $k^\epsilon(t) \geq 0$.
\end{enumerate}

\end{lemme}

\begin{proof}
We will write $T^\epsilon(t) = (\cos \alpha^\epsilon(t), \sin \alpha^\epsilon(t))$ and $s^\epsilon$ a parametrization by arc length of $c^\epsilon$ (such that $s^\epsilon(0) = 0$ and $ds^\epsilon/dt = \norm{dc^\epsilon/dt}$).

Since the curvature of $c$ at $t=0$ is nonzero, we may assume that the angle $t \mapsto \alpha(t)$ is decreasing on the interval $(-m, m)$ (reducing $m$ if necessary). We may also assume that $\alpha(t) \in (0, \pi)$.

We have:
\[ T^\epsilon(t) = \frac{(\cos (\alpha^1(t)), \epsilon \sin (\alpha^1(t)))}{\sqrt{\cos^2 (\alpha^1(t)) + \epsilon^2 \sin^2 (\alpha^1(t))}}. \]

Thus, \begin{equation} \label{eq1} \tan (\alpha^\epsilon(t)) = \epsilon \tan \alpha^1(t) \end{equation} for $t \in (0, m)$.

In particular, $t \mapsto \alpha^\epsilon(t)$ is decreasing (thus $k^\epsilon(t) \geq 0$), so for $\epsilon$ small enough, there exists $t_c \in (0, m)$ such that $\sin \alpha^\epsilon(t_c) = m$. Thus for all $t \in (t_c, m)$, we have $T^\epsilon_z(t) \leq m$, whereas for all $t \in (0, t_c)$, we have $T^\epsilon_z(t) \geq m$.

We will now show that $T^\epsilon_z(t) \geq m$ implies $k^\epsilon(t) \geq m^4/\epsilon^2$, which will conclude the proof of the lemma.

After differentiation of~(\ref{eq1}), we obtain
\[ \begin{aligned} \frac{d\alpha^\epsilon}{dt} (1 + \tan^2 (\alpha^\epsilon(t))) & = \epsilon \frac{d\alpha^1}{dt} \frac{1}{\cos^2 (\alpha^1(t))}
\\ \frac{d\alpha^\epsilon}{dt} (1 + \epsilon^2 \tan^2 (\alpha^1(t))) & = \epsilon \frac{d\alpha^1}{dt} \frac{1}{\cos^2 (\alpha^1(t))}
\\ \frac{d\alpha^\epsilon}{dt} & = \frac{d\alpha^1}{dt} \frac{\epsilon}{\cos^2 (\alpha^1(t)) + \epsilon^2 \sin^2 (\alpha^1(t))}
\\ \frac{d\alpha^\epsilon}{dt} & = \frac{d\alpha^1}{dt} \frac{\epsilon}{(ds^\epsilon/dt)^2}
\end{aligned} \]

Thus the curvature of $c_\epsilon$ is
\[ k^\epsilon(t) = \frac{d\alpha^\epsilon}{ds} = \frac{d\alpha^\epsilon}{dt} \cdot \frac{1}{dt/ds^\epsilon} = k^1(t) \cdot \epsilon \cdot \frac{1}{(dt/ds^\epsilon)^3}. \]

Assuming that $T^\epsilon_z \geq m$, we obtain
\[ \frac{\epsilon \sin \alpha}{ds^\epsilon/dt} = T^\epsilon_z \geq m \]
and thus
\[ \frac{ds^\epsilon}{dt} \leq \frac{\epsilon}{m}. \]

Finally, 
\[ k^\epsilon(t) \geq k^1(t) \cdot \frac{m^3}{\epsilon^2} \geq \frac{m^4}{\epsilon^2}. \]
\end{proof}

\section{Curvature in a flattened tube} \label{sectCurvatureInTube}
In sections~\ref{sectCurvatureInTube}, \ref{sectDynamicsInTube} and~\ref{sectEndProof}, we choose constants $\nu$, $m$, $\delta$, and $\epsilon$, in the interval $(0,1)$, such that:
\begin{enumerate}
\item These constants are sufficiently small: how small they need to be depends on the choice of the billiard $D$ and the surface $\Sigma$.
\item These constants satisfy $\nu \gg m \gg \delta \gg \epsilon$. This means that the ratios $m / \nu$, $\delta / m$ and $\epsilon / \delta$ are sufficiently small, again with respect to the choice of $D$ and $\Sigma$. We even assume that the ratios $m / \nu^{1000}$, $\delta / m^{1000}$ and $\epsilon / \delta^{1000}$ are small.
\end{enumerate}

We consider these constants as fixed once and for all, to avoid adding in each statement a prefix such as ``there exists $\nu_0 > 0$, such that for all $\nu \leq \nu_0$, there exists $m_0 > 0$, such that for all $m \leq m_0$\dots''.


In this section, we consider a circular obstacle $\Delta_{i_0}$, with center $q_0 \in \mathbb S^2$, and use the stereographic projection of $\mathbb S^3$ with $q_0$ as the south pole (that is, $q_0$ has coordinates $(0, 0, 0)$). We will use cylindric coordinates $(r, \theta, z)$. The circle $\Delta_{i_0}$ is defined by the equation $r = R$, $z = 0$, where $R \in (0,1)$.

Consider $\tilde {\mathcal T} = \setof{(r, \theta, z) \in \Sigma_\epsilon}{r \leq R + \delta}$, and $\mathcal T$ the connected component of $\tilde{\mathcal T}$ which contains $\partial \Delta_{i_0}$. The ``tube'' $\mathcal T$ is a surface of revolution (assumption~3 of Theorem~\ref{thmAnosov}), obtained by rotation of a curve $\mathcal S$ around the $z$-axis.

More precisely, we define the curve $\mathcal S$ as the intersection of $\mathcal T$ with the half great sphere corresponding to the equation $\theta = 0$. It has an upper part described by the equation $z = h(r)$, and a lower part given by $z = -h(r)$. Here, the mapping $h$ is nonnegative, defined continuously on the interval $[R, R + \delta]$, smooth on $(R, R + \delta)$, and such that $h(R) = 0$ (see Figure~\ref{figStereo}).

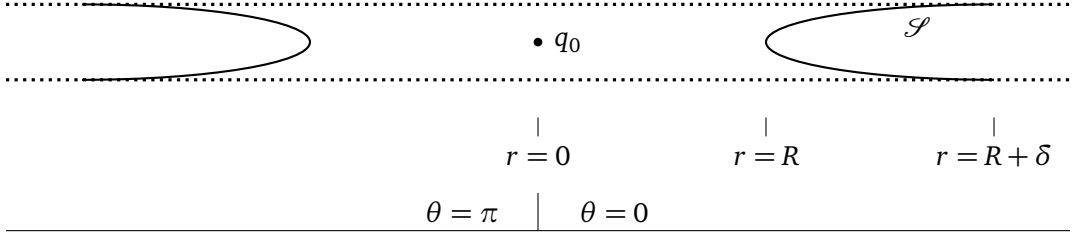
\begin{figure}[!ht]
\begin{center}
\begin{tikzpicture}
    \draw[thick] (-6,0) [partial ellipse=-90:90:3cm and 0.5cm];
    \draw[thick] (6,0) [partial ellipse=90:270:3cm and 0.5cm];
    \draw[dotted, very thick] (-7,0.5) -- (7,0.5);
    \draw[dotted, very thick] (-7,-0.5) -- (7,-0.5);
    \draw (6,-1) -- (6,-1.25);
    \draw (3,-1) -- (3,-1.25);
    \draw (0,-1) -- (0,-1.25);
    \node at (6, -1.5) {$r = R + \delta$};
    \node at (3, -1.5) {$r = R$};
    \node at (0, -1.5) {$r = 0$};
    
    \node at (5, 0.2) {$\mathcal S$};
    \node at (0.4, 0) {$q_0$};
    \draw[fill=black] (0,0) circle (0.05cm);
    
    \draw (0,-2) -- (0,-2.5);
    \draw (-7,-2.5) -- (7,-2.5);
    \node at (1, -2.25) {$\theta = 0$};
    \node at (-1, -2.25) {$\theta = \pi$};
\end{tikzpicture}
\end{center}

\caption{The curve $\mathcal S$ (on the right-hand side) seen in stereographic projection. The two dotted lines correspond to two spheres in $\mathbb S^3$ which are close to the great sphere $\mathbb S^2$.} \label{figStereo}
\end{figure}

We consider the Euclidean metric $g_\mathrm{eucl}$ and the metric of the stereographic projection \[ g_\mathrm{stereo} = \xi^2 g_\mathrm{eucl}, \quad \text{ where } \quad \xi = \frac{2}{(1 + r^2 + z^2)}. \] The Euclidean scalar product is denoted by $\prodscal{\cdot}{\cdot} = g_\mathrm{eucl}(\cdot, \cdot)$, and the Euclidean norm is $\norm{\cdot} = \sqrt{\prodscal{\cdot}{\cdot}}$. The Levi-Civita connection of $g_\mathrm{stereo}$ is written $\nabla$.

There are three unit vectors $e_r(q)$, $e_\theta(q)$ and $e_z(q)$ for each $q = (x, y, z) \in \mathbb R^3 \setminus \{0\}$, where (in cartesian coordinates) $e_r(q) = (x / \sqrt{x^2 + y^2}, y / \sqrt{x^2 + y^2}, 0)$, $e_z = (0,0,1)$ and $e_\theta = e_z \times e_r$. The Euclidean norm of these vectors is $1$. If $p$ is a vector in $\mathbb R^3 \setminus \{0\}$, we will write $p_r = \prodscal{p}{e_r}$, $p_\theta = \prodscal{p}{e_\theta}$ and $p_z = \prodscal{p}{e_z}$.

The second fundamental form of the surface $\Sigma_\epsilon$ is defined by $\mathrm{II}_q(v) = \xi^2 \prodscal{\nabla_v v}{N(q)}$, where $N(q)$ is the unit normal vector to $\Sigma_\epsilon$ for the metric $g_\mathrm{stereo}$ (thus $\xi \norm{N} = 1$). Any geodesic ($q(t), p(t)$) on the tube satisfies the equation
\[ \nabla_p p = \mathrm{II}_q(p) N(q). \] When studying a tube, we always assume that $N(q)$ points to the outside of the tube, and write $N_r$, $N_\theta$ and $N_z$ its spherical coordinates.

By symmetry, at each point $q$, the two principal curvatures are $k_1 (q) = \mathrm{II}_q(e_\theta/\xi)$ and $k_2 = \mathrm{II}_q(e_s/\xi)$, where $e_s$ is a unit vector which is orthogonal to $e_\theta$ and tangent to $\Sigma_\epsilon$.

The curvature of $\Sigma$ at $q$ is $k_1 k_2 + 1$.

\begin{lemme} \label{lemmak1}
Consider the waist $W$ of the tube $\mathcal T$ (the smallest horizontal circle contained in the tube), and $\kappa$ its curvature for the metric $g_\mathrm{stereo}$. Then for all $q = (q_r, q_\theta, q_z)$ in the tube:
\begin{enumerate}
\item $\abs{k_1(q) - \xi N_r(q) \kappa} \leq m^2$;
\item $k_1(q) \leq 0$.
\end{enumerate}
\end{lemme}
\begin{proof}
We consider a parametrization $\mathcal C(t)$ of the horizontal circle contained in $\Sigma$ such that $\mathcal C(0) = q$, with unit speed (for $g_\mathrm{stereo}$). The principal curvature $k_1$ is
\[ k_1 = \xi^2 \prodscal{\nabla_{\mathcal C'(0)} \mathcal C'(0)}{N(q)}. \]
Since the circle $\mathcal C$ is close to the circle $W$, we have
\[ \norm{\nabla_{\mathcal C'(0)} \mathcal C'(0) - \kappa \frac{e_r}{\xi}} \leq m^3 \]
and thus
\[ \abs{k_1(q) - N_r(q) \kappa} = \abs{\xi^2 \prodscal{\nabla_{\mathcal C'(0)} \mathcal C'(0)}{N(q)} - \xi \kappa N_r(q)} \leq m^3 \xi^2 \leq m^2 \]
which proves the first statement.

We now prove the second statement: by symmetry, we may assume that $q_z \geq 0$. Writing $\nabla_{\mathcal C'(0)} \mathcal C'(0) = (r, \theta, z)$, we obtain $k_1 = \xi^2(r N_r(q) + z N_z(q))$. Notice that $r \leq 0$ and $z \geq 0$. At the same time, $N_r(q) \geq 0$ and $N_z(q) \leq 0$. Thus, $k_1 \leq 0$.
\end{proof}

\begin{lemme} \label{lemmaChristoffel}
Consider a smooth curve $q(t)$ in $\Sigma_\epsilon$ and $p(t) = \dot q(t)$. Then:
\[ \abs{\nabla_{p(t)} p(t) - \dot p(t)} \leq \norm{p(t)}^2 / m^{1/10}. \]
\end{lemme}
\begin{proof}
\[ \nabla_{p(t)} p(t) = \sum_{1 \leq i, j, k \leq 3} p_i(t) p_j(t) \Gamma_{ij}^k e_k + \dot p(t), \]
where $p_1$, $p_2$, $p_3$ mean respectively $p_r$, $p_\theta$ and $p_z$, and $e_1$, $e_2$, $e_3$ mean respectively $e_r$, $e_\theta$ and $e_z$, and $\Gamma_{ij}^k$ are the Christoffel symbols of $\nabla$.

Thus
\[ \abs{\nabla_{p(t)} p(t) - \dot p(t)} \leq \sum_{1 \leq i, j, k \leq 3} \norm{p(t)}^2 \abs{\Gamma_{ij}^k} \leq \norm{p(t)}^2 / m^{1/10}. \]
\end{proof}

\begin{lemme} \label{lemmak2}
Consider the curve $c: [a,b] \to \mathbb R^2$, $c(t) = (r(t), z(t))$, which is a unit-speed parametrization (for the metric $g_\mathrm{stereo}$) of the curve $\mathcal S$. Assume that $0 \in (a, b)$, $c(0) = R$ and $c'(0) = (0,1)$. Recall that the curvature of $c$ at $0$ is nonzero. For all $t \in (a, b)$, we have $r(t) \geq R$, $r(-t) = r(t)$ and $z(-t) = -z(t)$. The unit tangent vector to $c(t)$ is $T(t) = (T_r(t), T_z(t))$ (that is, $g_\mathrm{eucl}(T, T) = 1$), and the normal vector is $N = (T_z(t), -T_r(t))$. The curvature of $c$ for the metric $g_\mathrm{stereo}$ is written $k_\mathrm{stereo}$.

Then there exists $t_c$ such that
\begin{enumerate}
\item for all $t \in (0, t_c)$, $T_z(t) \geq m / 2$ and $k_\mathrm{stereo}(t) \geq m^5 / \epsilon^2$,
\item for all $t \in (t_c, m)$, $T_z(t) \leq 2 m$ and $k_\mathrm{stereo}(t) \geq - 1 / m^{1/4}$.
\item $\int_a^b k_\mathrm{stereo} (t) dt \leq 1/m^{1/8}$.
\end{enumerate}

In the following, we will write $r_c = r(t_c)$.

\end{lemme}
\begin{proof}
The strategy is to reduce the problem to a flat one, and then apply Lemma~\ref{lemmaHighCurvatureFlat}.

Consider the mapping $\Phi : (0,1) \times (- \sqrt{\epsilon}, \sqrt{\epsilon}) \to \mathbb R^2$, whose restriction to the line $z = 0$ is the identity, which is a bijection onto its image, and such that $\Phi^{-1} \circ f_\epsilon \circ \Phi$ coincides with the affine map $A: (r, z) \mapsto (r, \epsilon z)$ on its domain.

Then
\[ \sup_{\mathbb R \times (- \sqrt{\epsilon}, \sqrt{\epsilon})} \norm{\Phi - \mathrm{Id}} < m^2 \quad \text{ and } \quad \sup_{\mathbb R \times (- \sqrt{\epsilon}, \sqrt{\epsilon})} \norm{D\Phi - D\mathrm{Id}} < m^2 \]

(``$\Phi$ is $C^1$-close to the identity''). We will write $\tilde c(t) = \Phi^{-1} \circ c(t)$. Notice that $\tilde c$ is a ``flattened curve'' in the Euclidean sense, and thus, Lemma~\ref{lemmaHighCurvatureFlat} applies to $\tilde c$.

We denote by $s$ the arc length of $c$ for the metric $g_\mathrm{eucl}$, and consider the curvature $k_\mathrm{eucl}$ of $c$ for the metric $g_\mathrm{eucl}$. Similarly, we denote by $\tilde s$ the arc length of $\tilde c$ for the metric $g_\mathrm{eucl}$, and consider the curvature $\tilde k_\mathrm{eucl}$ of $\tilde c$ for the metric $g_\mathrm{eucl}$. Also, $\tilde T$ (\emph{resp.} $\tilde N$) is the unit tangent (\emph{resp.} normal) vector to $\tilde c$ for the metric $g_\mathrm{eucl}$. Then:

\[ \begin{aligned}
k_\mathrm{eucl} & = \prodscal{\frac{d \left(D\Phi (\tilde c(t)) \cdot \frac{\tilde T}{\norm{D\Phi (\tilde c(t)) \cdot \tilde T}}\right)}{d\tilde s} \cdot \frac{d \tilde s}{ds}}{N}
\\ & = \prodscal{D^2 \Phi (c(t)) \cdot \left(\frac{\tilde T}{\norm{D\Phi (\tilde c(t)) \cdot \tilde T}}, \tilde T \right) \cdot \frac{d \tilde s}{ds}}{N}
\\ & + \prodscal{D\Phi (\tilde c(t)) \cdot \left( \frac{d\tilde T}{d\tilde s} \cdot \frac{1}{\norm{D\Phi (\tilde c(t)) \cdot \tilde T}} - \tilde T \frac{d}{ds}\left(\frac{1}{\norm{D\Phi (\tilde c(t)) \cdot \tilde T}} \right)\right) \cdot \frac{d \tilde s}{ds}}{N}
\\ & \geq m^{1/10} \tilde k_\mathrm{eucl} (t) - \frac{1}{m^{1/10}}.
\end{aligned}
\]

On the other hand, using Lemma~\ref{lemmaChristoffel}:

\[ \begin{aligned} \abs{k_\mathrm{stereo} - \frac{1}{\xi} k_\mathrm{eucl}} & = \abs{\xi^2 \prodscal{\nabla_{T/\xi} T/\xi}{N/\xi} - \frac{1}{\xi} \prodscal{\frac{dT}{ds}}{N}}
\\ & = \abs{\frac{1}{\xi} \prodscal{\nabla_T T - \frac{dT}{ds}}{N}}
\\ & \leq \frac{1}{\xi m^{1/10}}. \end{aligned} \]

Hence
\[ k_\mathrm{stereo}(t) \geq m^{1/9} \cdot k_\mathrm{eucl}(t) - \frac{1}{m^{1/9}}. \]

Finally,
\[ k_\mathrm{stereo}(t) \geq m^{1/4} \cdot \tilde k_\mathrm{eucl}(t) - \frac{1}{m^{1/4}} \]
and thus Lemma~\ref{lemmaHighCurvatureFlat} applied to the curve $\tilde c$ allows us to prove Statements 1 and 2.

We now prove Statement 3. From the inequality
\[ \abs{k_\mathrm{stereo} - \frac{1}{\xi} k_\mathrm{eucl}} \leq \frac{1}{\xi m^{1/10}}
\]
we also obtain
\[ k_\mathrm{stereo} \leq \frac{1}{\xi} k_\mathrm{eucl} + \frac{1}{\xi m^{1/10}} \]
and thus
\[ \begin{aligned} \int_a^b k_\mathrm{stereo} (t) dt & \leq \int_a^b (k_\mathrm{eucl} (t) + \frac{1}{m^{1/10}}) \frac{ds}{dt} dt \\ & \leq \left (\int_a^b 2\frac{d\alpha}{ds} ds + (b-a)\frac{1}{m^{1/9}} \right) \\ & \leq 4\pi + \frac{(b-a)}{m^{1/9}} \leq 1/m^{1/8}. \end{aligned} \]
\end{proof}

\section{The dynamics in the tube} \label{sectDynamicsInTube}

In this section, we consider a unit speed geodesic $(q(t), p(t))$ in $T^1 \mathcal{T}$, where $q$ is the position and $p$ is the speed. We will write $(q_r, q_\theta, q_z)$ and $(p_r, p_\theta, p_z)$ the cylindric coordinates of $q$ and $p$. Also, we write $p_s = \sqrt{p_r^2 + p_z^2}$ and define $e_s$ as the unit vector such that $p_s = \prodscal{p}{e_s}$.

The field $r e_\theta$ is a Killing field on $\mathcal T$. Thus, the quantity $L = g(r e_\theta, p) = \xi^2 r p_\theta$ is constant on each geodesic (this is the Clairaut first integral).

The geodesic starts on the boundary of the tube at $t = t^\mathrm{in}$ and exits at $t = t^\mathrm{out}$ (if the geodesic does not exit the tube, we write $t^\mathrm{out} = +\infty$).

\begin{lemme} \label{lemmaNonGrazingPs}
For all $t \in [t^\mathrm{in}, t^\mathrm{out}]$, we have
\[ \abs{p_s(t) - p_s(t^\mathrm{in})} \leq m^{10}. \]
\end{lemme}
\begin{proof}
For all $t$, \[ p_s^2 = 1/\xi^2 - p_\theta^2 = 1/\xi^2 - \frac{L^2}{\xi^4 r^2}. \]

The coordinate $r$ varies between $R$ and $R + \delta$. Moreover, knowing that $z \leq \delta$ (with $\epsilon$ sufficiently small), the quantity $\xi$ varies between $\frac{2}{1 + (R+\delta)^2 + \delta^2}$ and $\frac{2}{1 + R^2}$. Thus the variation of $p_s^2$ is less than $m^{10}$.
\end{proof}

\begin{lemme} \label{lemmaTimeNG}
Assume that $\abs{p_s(t^\mathrm{in})} \geq m$. Then the time spent in the tube is smaller than $6 \delta / m$.
\end{lemme}
\begin{proof}
The length of the curve $c^\epsilon$ is smaller than $3\delta$. Moreover, Lemma~\ref{lemmaNonGrazingPs} implies that $\abs{p_s(t)} \geq m/2$ for all $t$. In particular, $ds/dt$ does not change sign. Thus, the time spent in the tube is smaller than $3\delta / (ds/dt) \leq 6 \delta/m$.
\end{proof}

%
%
%

\begin{lemme} \label{lemmaNonGrazingK}
Assume that $\abs{p_s(t^\mathrm{in})} \geq m$. Then
\[ \abs{ \int_{t^\mathrm{in}}^{t^\mathrm{out}} K(t) dt - \frac{2 \kappa}{\xi(t^\mathrm{in}) p_s(t^\mathrm{in})} } \leq m^{1/3}. \]
\end{lemme}
\begin{proof}
Let us divide the problem into several steps by using the triangle inequality (the integrals are taken between the times $t^\mathrm{in}$ and $t^\mathrm{out}$).

\[ \begin{aligned} & \abs{ \int_{t^\mathrm{in}}^{t^\mathrm{out}} K(t) dt - \frac{2 \kappa}{\xi(t^\mathrm{in}) p_s(t^\mathrm{in})} }
\\ & \leq \abs{\int K(t)dt - \int k_1 k_2 dt} + \abs{\int k_1 k_2 dt - \int k_2 \xi N_r \kappa dt}
\\ & + \abs{\int k_2 \xi N_r \kappa dt - \int k_2 \xi \prodscal{N}{e_r(t^\mathrm{in})} \kappa dt}
\\ & + \abs{\int k_2 \xi \prodscal{N}{e_r(t^\mathrm{in})} \kappa dt - \int \frac{1}{\xi(t) p_s(t)^2} \mathrm{II}_q(p) \prodscal{N}{e_r(t^\mathrm{in})} \kappa dt}
\\ & + \abs{\int \frac{1}{\xi(t) p_s(t)^2} \mathrm{II}_q(p) \prodscal{N}{e_r(t^\mathrm{in})} \kappa dt - \int \frac{1}{\xi(t) p_s(t)^2} \prodscal{\dot p}{e_r(t^\mathrm{in})} \kappa dt}
\\ & + \abs{\int \frac{1}{\xi(t) p_s(t)^2} \prodscal{\dot p}{e_r(t^\mathrm{in})} \kappa dt - \int \frac{1}{\xi(t^\mathrm{in}) p_s(t^\mathrm{in})^2} \prodscal{\dot p}{e_r(t^\mathrm{in})} \kappa dt}
\\ & + \abs{\int \frac{1}{\xi(t^\mathrm{in}) p_s(t^\mathrm{in})^2} \prodscal{\dot p}{e_r(t^\mathrm{in})} \kappa dt - \frac{2\kappa}{\xi(t^\mathrm{in}) p_s(t^\mathrm{in})}}.
\end{aligned} \]

We will now show that each of the terms is smaller than $\sqrt{m}$.

\begin{enumerate}
\item With Lemma~\ref{lemmaTimeNG},
\[ \abs{\int K(t)dt - \int k_1 k_2 dt} = \abs{\int 1 dt} \leq 6\delta / m \]
\item With Lemma~\ref{lemmak1},
\[ \abs{\int k_1 k_2 dt - \int k_2 \xi N_r \kappa dt} \leq m\int k_2 dt = m \int \frac{k_2}{p_s} ds \leq 2 \int k_2 ds \]
Moreover, with Lemma~\ref{lemmak2}, $\int k_2 ds \leq 1 / m^{1/8}$, so
\[ \abs{\int K(t) - \int k_2 \xi N_r \kappa dt} \leq \sqrt{m}. \]
\item With Lemma~\ref{lemmaTimeNG}, we know that $e_r(t) - e_r(t^\mathrm{in}) \leq \sqrt{\delta}$. Thus:
\[ \begin{aligned} \abs{\int k_2 \xi N_r \kappa dt - \int k_2 \xi \prodscal{N}{e_r(t^\mathrm{in})} \kappa dt} & = \abs{\int k_2 \xi \prodscal{N}{e_r(t) - e_r(t^\mathrm{in})} \kappa dt} \\ & \leq \delta^{1/4} \int k_2 ds \\ & \leq \sqrt{m}. \end{aligned} \]
\item Using the fact that $\mathrm{II}_q(p) = \xi^2 (k_1 p_\theta^2 + k_2 p_s^2)$, we obtain:
\[ \begin{aligned} ~ & \abs{\int k_2 \xi \prodscal{N}{e_r(t^\mathrm{in})} \kappa dt - \int \frac{1}{\xi(t) p_s(t)^2} \mathrm{II}_q(p) \prodscal{N}{e_r(t^\mathrm{in})} \kappa dt} \\ & = \abs{\int k_1 \frac{p_\theta^2}{p_s^2} \xi \prodscal{N}{e_r(t^\mathrm{in})} \kappa dt} \\ & \leq \abs{\int \xi \kappa(1 + m) \frac{1}{\xi^2(m/2)^2} \xi \kappa dt} \leq \sqrt{m} \end{aligned} \]
\item With Lemma~\ref{lemmaChristoffel}, since $\mathrm{II}_q(p)N = \nabla_p p$, we have:
\[ \begin{aligned} & \abs{\int \frac{1}{\xi(t) p_s(t)^2} \mathrm{II}_q(p) \prodscal{N}{e_r(t^\mathrm{in})} \kappa dt - \int \frac{1}{\xi(t) p_s(t)^2} \prodscal{\dot p}{e_r(t^\mathrm{in})} \kappa dt}
\\ & \leq \int \frac{\kappa}{\xi(m/2)^2 m}dt \leq \sqrt{m} \end{aligned} \]
\item We use Lemma~\ref{lemmaNonGrazingPs}:
\[ \begin{aligned} & \abs{\int \frac{1}{\xi(t) p_s(t)^2} \prodscal{\dot p}{e_r(t^\mathrm{in})} \kappa dt - \int \frac{1}{\xi(t^\mathrm{in}) p_s(t^\mathrm{in})^2} \prodscal{\dot p}{e_r(t^\mathrm{in})} \kappa dt} \\ & \leq \abs{m^{2/3} \int \prodscal{\dot p}{e_r(t^\mathrm{in})}} \leq m^{2/3} \norm{p(t^\mathrm{out}) - p(t^\mathrm{in})} \leq \sqrt{m} \end{aligned} \]
\item Since the tube is symmetric, we have $\prodscal{p_r(t^\mathrm{in})}{e_r(t^\mathrm{in})} = - \prodscal{p_r(t^\mathrm{out})}{e_r(t^\mathrm{out})}$. Thus:
\[ \begin{aligned} & \abs{\int \frac{1}{\xi(t^\mathrm{in}) p_s(t^\mathrm{in})^2} \prodscal{\dot p}{e_r(t^\mathrm{in})} \kappa dt - \frac{2\kappa}{\xi(t^\mathrm{in}) p_s(t^\mathrm{in})}}
\\ & = \abs{\frac{1}{\xi(t^\mathrm{in}) p_s(t^\mathrm{in})^2} \kappa} \abs{ \prodscal{p(t^\mathrm{out}) - p(t^\mathrm{in})}{e_r(t^\mathrm{in})} - 2 \prodscal{p(t^\mathrm{in})}{e_s(t^\mathrm{in})}}
\\ & \leq \frac{2}{m} \kappa m^2 \leq \sqrt{m}. \end{aligned} \]
\end{enumerate}
\end{proof}


In the following lemma, we consider the constant $r_c$ given by Lemma~\ref{lemmak2}.

\begin{lemme} \label{lemmaTimeRc}
In the tube, the time during which $r \geq r_c$ is smaller than $\delta^{1/3}$ (in other words, $\int_{r \geq r_c} dt \leq \delta^{1/3}$).
\end{lemme}
\begin{proof}
First, we compute for $r \geq r_c$:
\[ \begin{aligned} \frac{dr}{dt} & = p_r
\\ & = \pm N_z \xi p_s
\\ & = \pm N_z \xi \sqrt{\frac{1}{\xi^2} - p_\theta^2}
\\ & = \pm N_z \xi \sqrt{\frac{1}{\xi^2} - \frac{L^2}{\xi^4 r^2}}
\\ & = \pm N_z \sqrt{1 - L^2 f(r)} \end{aligned} \]
where $f(r) = \frac{(1 + r^2 + h(r)^2)^2}{4r^2}$.

We compute:
\[ \begin{aligned} f'(r) & = 2(1 + r^2 + h(r)^2) (r^2 - 1 + h(r) (2h'(r) - h(r)))/r^3
\\ & \leq - \delta^{1/10}.  \end{aligned} \]

Thus for all $r_0 \geq r_c$ we may write:
\[ f(r) \leq f(r_0) - \delta^{1/10} (r-r_0). \]

Moreover, $f(r) \leq \delta^{-1/20}$.

\subparagraph{First case.} We assume that there exists $r_0 \geq r_c$ such that $1 - L^2 f(r_0) = 0$. Then the geodesic has the following life: $r$ decreases from $R + \delta$ to $r_0$, reaches $r_0$ at some time $t_0$, and then increases from $r_0$ to $R + \delta$. In this case, the length of the geodesic is
\[ \begin{aligned} 2 (t_0 - t^\mathrm{in})
& = 2 \int_{t^\mathrm{in}}^{t_0} dt
\\ & = 2 \int_{r_0}^{R + \delta} \frac{1}{- dr / dt} dr
\\ & = 2\int_{r_0}^{R + \delta} \frac{1}{\abs{N_z} \sqrt{1 - L^2 f(r)}} dr
\\ & \leq 2 \int_{r_0}^{R + \delta} \frac{2}{\sqrt{1 - L^2 f(r_0) - L^2 \delta^{1/10} (r - r_0)}} dr
\\ & \leq 4 \int_{r_0}^{R + \delta} \frac{1}{\sqrt{L^2 \delta^{1/10} (r - r_0)}} dr
\\ & \leq \frac{8 \sqrt{R + \delta - r_0}}{L \delta^{1/20}} \end{aligned}
\]
Since $1 - L^2 f(r_0) = 0$, we have $L = 1 / f(r_0)$ and thus $L \geq \delta^{1/20}$. Hence \[ 2 (t^\mathrm{out} - t_0) \leq 8 \frac{\sqrt{\delta}}{\delta^{1/10}} \leq \delta^{1/3}. \]

\paragraph{Second case.} We assume that $1 - L^2 f(r) > 0$ for all $r \in (r_c, R + \delta)$. Then the geodesic goes through the zone $r \geq r_c$ and enters the zone $r \leq r_c$ at some time $t_c$. Then either it remains in this zone for all times, or it exits this zone and goes through the zone $r \geq r_c$ once more. Thus, the time spent in the zone $r \geq r_c$ is at most

\[ \begin{aligned} 2 \int_{t^\mathrm{in}}^{t} dt & = 2 \int_{r_c}^{R + \delta} \frac{1}{- dr / dt} dr
\\ & = 2\int_{r_0}^{R + \delta} \frac{1}{\abs{N_z} \sqrt{1 - L^2 f(r)}} dr.
\end{aligned}
\]

If $L \leq \delta^{1/20}$ then \[ 2 \int_{t^\mathrm{in}}^{t} dt \leq 4 \int_{r_c}^{R + \delta} \frac{1}{\sqrt{1 - \delta^{1/10} \delta^{-1/20}}} dr \leq \delta^{1/3} \] which concludes the proof. Now assuming that $L \geq \delta^{1/20}$, we compute:

\[ \begin{aligned}
2 \int_{t^\mathrm{in}}^{t} dt & \leq 4 \int_{r_c}^{R + \delta} \frac{1}{\sqrt{1 - L^2 f(r_c) - L^2 \delta^{1/10} (r - r_c)}} dr
\\ & \leq \frac{8 \sqrt{R + \delta - r_c}}{L \delta^{1/20}}
\\ & \leq \delta^{1/3}
\end{aligned}
\]

\end{proof}

\begin{lemme} \label{lemmaCurvatureAlmostNegative}
The Gauss curvature $K$ of $\Sigma_\epsilon$ satisfies $K \leq \frac{1}{m^{1/3}}$. Moreover, $K \leq -1/\epsilon$ in the zone $r \leq  r_c$.
\end{lemme}
\begin{proof}
In the zone $r \geq r_c$, we have $k_1 \leq 0$ (Lemma~\ref{lemmak1}) and $k_2 \geq -1/m^{1/4}$ (Lemma~\ref{lemmak2}), so that $K = k_1 k_2 + 1 \leq 1 / m^{1/3}$.

In the zone $r \leq r_c$, we have $k_1 \leq - \kappa m / 2$ (Lemmas~\ref{lemmak1} and~\ref{lemmak2}) and $k_2 \geq m^5/\epsilon^2$. Thus, in this zone $K = k_1 k_2 + 1 \leq - \kappa m^6 / (2\epsilon^2) + 1$. In particular, $K \leq -1/\epsilon$.
\end{proof}

%
%
%
%
%
%

\begin{lemme} \label{lemmaU}
In this lemma, we consider a geodesic $(q(t), p(t))_{t \in [t^\mathrm{in}, t^\mathrm{out}]}$ in the tube $\mathcal T$, but we do not assume that $q(t^\mathrm{in})$ or $q(t^\mathrm{out})$ is on the boundary of $\mathcal T$. Consider a solution $u$ of the Riccati equation $u'(t) = -K(t) - u(t)^2$ such that $\abs{u(t^\mathrm{in})} \leq 1/m^2$. Then
\begin{enumerate}
\item $u(t^\mathrm{out}) \geq u(t^\mathrm{in}) - m$;
\item if the time spent in the tube $\mathcal T$ is at least $m$, then $u(t^\mathrm{out}) \geq 1/m^2$.
\end{enumerate}
\end{lemme}
\begin{proof}
Let $t^1 = \sup\setof{t \in [t^\mathrm{in}, t^\mathrm{out}]}{ u(t) \geq 2/m^2}$ (if this set is empty, let $t^1 = t^\mathrm{in}$), and $t^2 = \inf\setof{t \in [t^1, t^\mathrm{out}]}{ u(t) \leq - 2/m^2}$ (if this set is empty, let $t^2 = t^\mathrm{out}$).

There is a (possibly empty) interval $(t^3, t^4) \subseteq (t^1, t^\mathrm{out})$ such that, for all $t \in (t^1, t^\mathrm{out})$, $r(t) < r_c$ if and only if $t \in (t^3, t^4)$.

Assume that $t^2 \leq t^3$. Then using Lemmas~\ref{lemmaCurvatureAlmostNegative} and~\ref{lemmaTimeRc}, we have:

\[ \begin{aligned} u(t^2) & = u(t^1) + \int_{t^1}^{t^2} - K(t) - u(t)^2 dt
\\ & \geq u(t^1) - \frac{\delta^{1/3}}{m^{1/3}} - \frac{4\delta^{1/3}}{m^4}
\\ & \geq u(t^1) - m/2 \end{aligned} \]
which contradicts the fact that $u(t^2) \leq -2/m^2$. Thus, $t^2 \geq t^3$ and $u(t^3) \geq u(t^1) - m/2$.

For $t \in (t^3, t^4)$, we have:
\[ u'(t) = -K(t) - u(t)^2 \geq \frac{1}{\epsilon} - 4 / m^4 \geq 1/m^5 \]
which implies that $t^2 \geq t^4$ and $u(t^4) \geq u(t^1) - m/2 + (t^4 - t^3)/m^5$. Moreover, $u(t^4) \leq 2 / m^2$ (because $t^1 \leq t^4$) so $t^4 - t^3 \leq m^2$. If $t^\mathrm{out} - t^\mathrm{in} \geq m$, this implies that $t^1 > t^\mathrm{in}$ and thus $u(t^1) \geq 2/m^2$.

Finally,
\[ \begin{aligned} u(t^2) & = u(t^4) + \int_{t^4}^{t^2} - K(t) - u(t)^2 dt
\\ & \geq u(t^1) - m/2 - \frac{\delta^{1/3}}{m^{1/3}} - \frac{4\delta^{1/3}}{m^4}
\\ & \geq u(t^1) - m \end{aligned} \]
and thus $t^2 = t^\mathrm{out}$ and $u(t^\mathrm{out}) \geq u(t^1) - m \geq u(t^1) - m$.

If the time spent in $\mathcal T$ is at least $m$, then $u(t^1) \geq 2/m^2$ and thus $u(t^\mathrm{out}) \geq 1/m^2$.
\end{proof}

\begin{lemme} \label{lemmaUNG}
Assume that $\abs{p_s(t^\mathrm{in})} \geq m$. Consider a solution $u$ of the Riccati equation $u'(t) = -K(t) - u(t)^2$ such that $\abs{u(t^\mathrm{in})} \leq 2/m^2$. Then:
\[ \abs{u(t^\mathrm{out}) - u(t^\mathrm{in}) - \frac{2 \kappa}{\xi(t^\mathrm{in}) p_s(t^\mathrm{in})}} \leq m^{1/4}. \]
\end{lemme}
\begin{proof}
Let $t^1 = \inf\setof{t \in [t^\mathrm{in}, t^\mathrm{out}]}{\abs{u(t^\mathrm{in})} \geq 2/m^2}$ (if this set is empty, let $t^1 = t^\mathrm{out}$). We write $K = K^+ - K^-$, where $K^+ = \max (K, 0)$ is the positive part of $K$ and $K^- = max (-K, 0)$ is the negative part. Then, using Lemmas~\ref{lemmaCurvatureAlmostNegative}, \ref{lemmaNonGrazingK} and~\ref{lemmaTimeNG},

\[ \begin{aligned} u(t^1) & = u(t^\mathrm{in}) + \int_{t^\mathrm{in}}^{t^1} -K(t) - u(t)^2 dt
\\ \abs{u(t^1)} & \leq \abs{u(t^\mathrm{in})} + \int_{t^\mathrm{in}}^{t^\mathrm{out}} \abs{K(t)} + \abs{u(t)}^2 dt
\\ \abs{u(t^1)} & \leq \abs{u(t^\mathrm{in})} + 2 \int_{t^\mathrm{in}}^{t^\mathrm{out}} K^+(t) dt - \int_{t^\mathrm{in}}^{t^\mathrm{out}} K(t) dt +  \int_{t^\mathrm{in}}^{t^\mathrm{out}} \abs{u(t)}^2 dt
\\ & \leq \frac{1}{m^2} + 2 \cdot \frac{6\delta}{m} \cdot \frac{1}{m^{1/3}} + \frac{2 \abs{\kappa}}{m} + m^{1/3} + \frac{6\delta}{m} \cdot \frac{4}{m^4}
\\ & < \frac{2}{m^2}
\end{aligned} \]

Thus $t^1 = t^\mathrm{out}$ and
\[ \small \begin{aligned} \abs{u(t^\mathrm{out}) - u(t^\mathrm{in}) - \frac{2 \kappa}{\xi(t^\mathrm{in}) p_s(t^\mathrm{in})}} & = \abs{u(t^\mathrm{out}) - u(t^\mathrm{in}) - \int_{t^\mathrm{in}}^{t^\mathrm{out}} K(t)dt} + \abs{\int_{t^\mathrm{in}}^{t^\mathrm{out}} K(t)dt -  \frac{2 \kappa}{\xi(t^\mathrm{in}) p_s(t^\mathrm{in})}} \\ & \leq \frac{6\delta}{m} \cdot \frac{4}{m^4} + m^{1/3}
\\ & \leq m^{1/4} \end{aligned} \]
\end{proof}


\section{End of the proof of Theorem~\ref{thmAnosov}} \label{sectEndProof}

\begin{thm} \label{lemmaAlmostGeodesic}
We say that a curve $\varphi : [a, b] \to D$ in the billiard $D$ is ``$\eta$-almost a geodesic'' if
\begin{enumerate}
\item $\norm{\phi'(t)}_g \leq 1$ for all $t \in [a, b]$,
\item $d(\varphi(b), \varphi(a)) \geq b - a - \eta$,
\end{enumerate}
where $d$ is the Riemannian distance and $g$ the Riemannian metric in the sphere $\mathbb S^2$.

Consider $H$ the horizon of $D$. Then there exists $\eta > 0$ such that for all $\eta$-almost geodesic $\varphi$, $a-b < H + \nu$.
\end{thm}
\begin{proof}
Assume that the conclusion is false. Then there exists a sequence $\varphi_n$ of $\frac{1}{n}$-almost geodesics such that $a = 0$ and $b = H + \nu$. By the Arzelà-Ascoli theorem, the sequence $\varphi_n$ converges in the $C^0$-topology to a curve $\varphi : [0,H+\nu] \to D$ which is a real geodesic in $D$. This contradicts the definition of $H$.
\end{proof}

\begin{proof}[End of the proof of Theorem~\ref{thmAnosov}]
Consider a geodesic $(q(t), p(t))_{t \in [0, H + 2 \nu]}$.

During its lifetime, the geodesic enters and exits the tubes. We will say that the tube is \emph{almost avoided} if the two following conditions are satisfied:
\begin{enumerate}
\item $\abs{p_s(t^\mathrm{in})} \leq m$;
\item $t^\mathrm{out} - t^\mathrm{in} \leq m$.
\end{enumerate}

\begin{lemme} \label{lemmaHorizon}
If the geodesic \emph{almost avoids} all the tubes in a time interval $(t^1, t^2)$, then $t^2 - t^1 \leq H + \nu$.
\end{lemme}
\begin{proof} The geodesic's projection $\pi \circ q$ is $\nu$-almost a geodesic in $\mathbb S^2$, so we may apply Lemma~\ref{lemmaAlmostGeodesic}.
\end{proof}

Now, consider an increasing sequence of times $(t_k)_{k \in \mathbb Z}$ such that:
\begin{enumerate}
\item For each $k \in \mathbb Z$, either $t_k$ is a time at which the geodesic exits a tube which is not \emph{almost avoided} (``type \textbf{A}''), or $t_k = t_{k-1} + H + 2 \nu$ (``type \textbf{B}'') ;
\item If the geodesic exits a tube which is not \emph{almost avoided} at a time $t^\mathrm{out}$, then there exists $k \in \mathbb Z$ such that $t_k = t^\mathrm{out}$.
\item For all $k \in \mathbb Z$, $\nu \leq t_{k+1} - t_k \leq H + 3 \nu$.
\end{enumerate}

According to Theorem~\ref{conditionConeRiccati}, we need to show that for any $k \in \mathbb Z$ and any $u$ solution of the Riccati equation along the geodesic $(p(t), q(t))$ with initial condition $u(t_k) = 0$, the solution $u$ is well-defined on $[t_k, t_{k+1}]$ and $u(t_{k+1}) > m$.

In the sphere, since the curvature is $1$, the geodesics follow the Riccati equation $u'(t) = -1 - u(t)^2$. If $q(t)$ remains outside the tubes in the time interval $(t^1, t^2)$, since the metric on $\Sigma_\epsilon$ is close to the metric of the Euclidean sphere, we have $u(t^2) \geq \tan (\arctan (u(t^1)) - t^2) - \nu$. This is also the case if one assumes that $q(t)$ \emph{almost avoids} all the tubes in this time interval (by Lemma~\ref{lemmaU}).

\paragraph{First case.} If $t_k$ is of type \textbf{A}, then consider the first time $t^\mathrm{in}$ (with $t^\mathrm{in} \in [t_k, t_{k+1}]$) at which the geodesic enters a tube which is not \emph{almost avoided} (such a time exists by Lemma~\ref{lemmaHorizon}). Then $u(t^\mathrm{in}) \geq - \tan (H + \nu) - \nu$. Then, by Lemmas~\ref{lemmaU} and~\ref{lemmaUNG}, \[ u(t_{k+1}) \geq u(t^\mathrm{in}) + 2 \tan (\pi/2 - A) - \nu \geq - \tan (H + \nu) + 2 \tan (\pi/2 - A) - 2\nu \geq m. \]

\paragraph{Second case.} If $t_k$ is of type \textbf{B}, notice that $q(t_k)$ is inside a tube which is not \emph{almost avoided}, because of Lemma~\ref{lemmaHorizon}. Therefore, the geodesic remains in the tube during the interval $[t_k + H + \nu, t_{k+1}]$. Since $t_{k+1} - (t_k + H + \nu) \geq \nu$, we may apply Lemma~\ref{lemmaU} and obtain: $u(t_{k+1}) \geq 1/m^2$.

Thus Theorem~\ref{conditionConeRiccati} applies, and Theorem~\ref{thmAnosov} is proved.
\end{proof}

\section{Embedding surfaces of genus at least 11} \label{sectGenus11}
Consider a billiard $D_R$ obtained from $12$ circles of equal radius $R$ whose centers are the vertices of a icosahedron which is inscribed in $\mathbb S^2$ (Figure~\ref{figBilliard}). The circles are disjoint if and only if $R < R_0 = \arctan(2)/2$. The horizon $H_R$ of the billiard $D_R$ depends on $R$.

\begin{prop}
\[ H_R \underset{R \to R_0}{\to} \pi - 2\arctan(2). \]
\end{prop}
\begin{proof}
Consider a sequence $R_n$ such that $R_n \underset{n \to \infty}{\to} R_0$, and a sequence $\gamma_n$ of portions of geodesics of $\mathbb S^2$ of maximal length, which are contained in $D_{R_n}$. Then there is a subsequence of $\gamma_n$ which converges uniformly to the portion of geodesic represented on Figure~\ref{figHorizon}, whose length is $\pi - 2\arctan(2)$.
\end{proof}

\begin{figure}[h!]
\centering
\includegraphics[width=200pt]{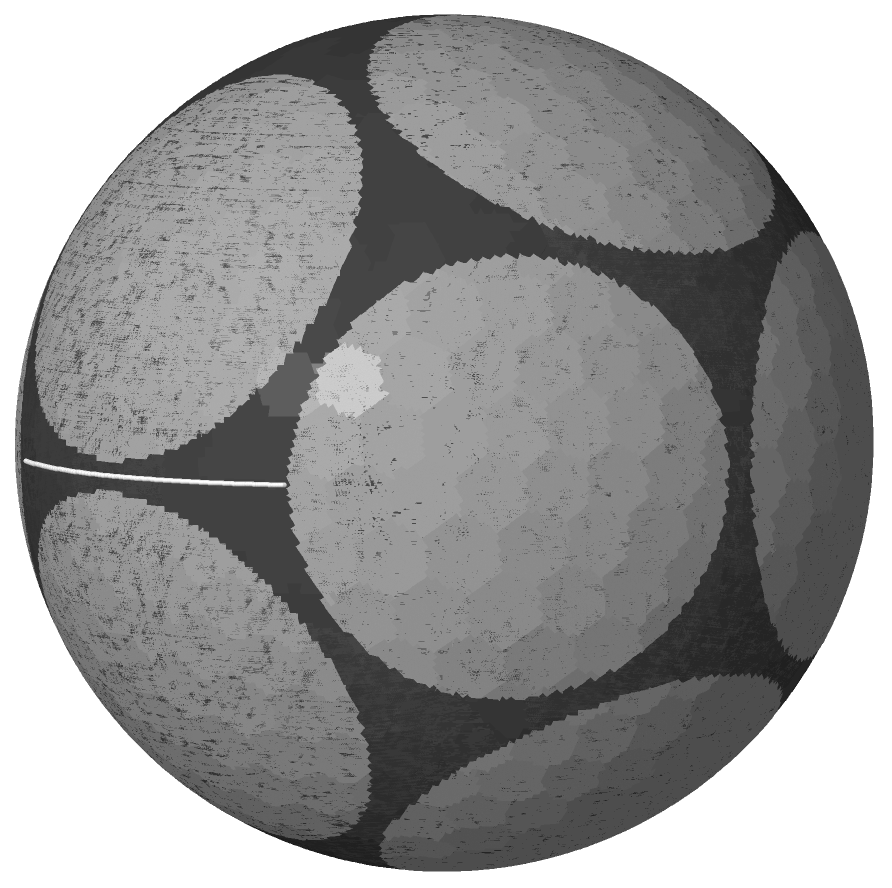}
\caption{The limit of a sequence of geodesics of maximum lengths in $D_{R_n}$.} \label{figHorizon}
\end{figure}

Thus \[ R + H_R \underset{R \to R_0}{\to} \pi - \frac{3}{2} \arctan (2) < \pi/2. \]
Therefore, for $R$ sufficiently close to $R_0$, we have $R + H_R < \pi/2$, and thus $2 \tan (\pi/2 - R) \geq \tan (\pi/2 - R) > \tan (H_R)$.

By applying Theorem~\ref{thmHypBilliard}, we obtain
\begin{cor}
The billiard $D$ is uniformly hyperbolic.
\end{cor}

To obtain a billiard with $n$ obstacles ($n \geq 12$), one may add spherical obstacles with small radii, which are disjoint from the others. Thus:
\begin{cor}
For any $n \geq 12$, there exists a spherical billiard with exactly $n$ circular obstacles, which is uniformly hyperbolic.
\end{cor}

Finally, Theorem~\ref{thmConstruction} completes the proof of Theorem~\ref{thmGenus11}.

\section{The Euclidean case: proof of Theorem~\ref{thmImpossible}} \label{sectImpossible}

We will study the geodesics in the ``tube'' $V_{i_0}^{\delta_1}$ for some fixed $i_0$. We assume (after rotation) that the center $q_0$ of the disk $\Delta_{i_0}$ is on the $z$-axis. The tube is a surface of revolution in $\mathbb R^3$, obtained by rotation of a curve $\gamma$ along the $z$-axis. We will use the cylindric coordinates $(r, \theta, z)$.

An essential difference with the spherical case is that the curve $\gamma$ is \emph{not} invariant by a symmetry with respect to a horizontal plane.

We assume that $\gamma(s)$ is parametrized by arc length, that $\gamma$ is in the half-plane $\{\theta = 0\}$, and that $\gamma(0) \in \partial \Delta_{i_0}$, with $\gamma'(0) = e_z$. Denote by $\alpha(s)$ the angle of the tangent vector $\gamma'(s)$ with the unit vector $e_r$, and consider the curvature $k(s) = \frac{d\alpha}{ds}$. Writing $\gamma(s) = (\gamma_r(s), \gamma_z(s))$, with the convention that the normal vector at $\gamma(0)$ is $e_r$, the principal curvatures of the surface $\Sigma$ at a point $\gamma(s)$ are $k_1 = -\frac{\cos (\alpha(s) - \pi/2)}{\gamma_r(s)}$ and $k_2 = - k(s)$; thus the Gauss curvature is
\[ K(s) = \frac{k(s)\cos (\alpha(s)-\pi/2)}{\gamma_r(s)} = \frac{\frac{d\alpha}{ds} \sin \alpha(s)}{\gamma_r(s)} = - \frac{d(\cos \alpha)/ds}{\gamma_r(s)}. \]

Assumption~2 of the theorem implies that $K(s) \geq 0$ and $\alpha(s) < - r_{i_0}/2$ for all $s \in (\delta_2, \delta_1)$. Moreover, since $\alpha(0) = \pi/2$, there exists $\delta_3 \in (0, \delta_2)$ such that $\alpha(\delta_3) = 0$. 

Consider a geodesic $(p(t), q(t))$ in $\Sigma$, which enters the tube at a time $-t_1$ and exits at time $t_1$. The symmetry assumption implies that the quantity $L = \prodscal{r(t) e_\theta(t)}{p(t)} = r(t) p_\theta(t)$ is constant on each unit speed geodesic. We assume that $L = \gamma_r(\delta_3)$ (the intermediate value theorem guarantees the existence of such a geodesic). We will consider $s(t)$ such that $(r(t), z(t)) = \gamma(s(t))$, and write $K(s)$ the Gauss curvature at $q(s(t))$. We have $p_s = \sqrt{1 - \frac{L^2}{r^2}}$; moreover, $s(0) = \delta_3$ and there is a unique time $t_2$ at which $s(t_2) = \delta_2$. 

\begin{lemme} The following estimate holds:
\[ \int_0^{t_2} K(t) dt \geq \frac{1-\cos (r_{i_0}/2)}{\sqrt{\gamma_r(\delta_2)^2 - \gamma_r(\delta_3)^2}}. \]
\end{lemme}
\begin{proof} We compute:
\[ \begin{aligned} \int_0^{t_2} K(t) dt & = \int_{\delta_3}^{\delta_2} \frac{K(s(t))}{ds/dt} ds
\\ & = - \int_{\delta_3}^{\delta_2} \frac{d(\cos \alpha)/ds}{\sqrt{\gamma_r(s)^2 - \gamma_r(\delta_3)^2}} ds
\\ & = \int_{\delta_3}^{\delta_2} g(s) \frac{d(\cos \alpha)}{ds} ds
\end{aligned} \]
where \[ g(s) = - \frac{1}{\sqrt{\gamma_r(s)^2 - \gamma_r(\delta_3)^2}}. \]

Notice that $g$ is differentiable on $(\delta_3, \delta_2)$ with $g'(s) > 0$, and that there exists $a > 0$ such that when $s$ tends to $\delta_3$,
\[ g(s) = - \frac{a}{\sqrt{s - \delta_3}} + o\left(\frac{1}{\sqrt{s - \delta_3}}\right). \]


Now integrating by parts,
\[ \small \begin{aligned} \int_{\delta_3}^{\delta_2} g(s) \frac{d(\cos \alpha)}{ds} ds = g(\delta_2) (\cos \alpha(\delta_2)-1) - \lim_{s \to \delta_3}  g(s) (\cos \alpha(s)-1) - \int_{\delta_3}^{\delta_2} g'(s) (\cos \alpha(s)-1) ds. \end{aligned} \]

Since $\lim_{s \to \delta_3} g(s) (\cos \alpha(s)-1) = 0$, and $g'(s) (\cos \alpha(s)-1) \leq 0$, we obtain:
\[ \begin{aligned} \int_0^{t_2} K(t) dt & \geq g(\delta_2) (\cos \alpha(\delta_2)-1)
\\ & \geq \frac{1-\cos (r_{i_0}/2)}{\sqrt{\gamma_r(\delta_2)^2 - \gamma_r(\delta_3)^2}}.
\end{aligned} \]
\end{proof}

The existence of conjugate points in the geodesic flow is given by the following lemma:

\begin{lemme} \label{conjugatePoints}
Consider the Riccati equation along the geodesic $(p(t), q(t))$:
\[ u(0) = 0, \quad \frac{du}{dt} = -K(t) - u(t)^2. \]
The solution of this equation blows up to $-\infty$ in finite positive time, and to $+\infty$ in finite negative time.
\end{lemme}
\begin{proof} First, notice that
\[ u(t_2) \leq - \int_{0}^{t_2} K(t) dt \leq - \frac{1-\cos (r_{i_0}/2)}{\sqrt{\gamma_r(\delta_2)^2 - \gamma_r(\delta_3)^2}}. \]

Moreover, for $t \in (t_2, t_1)$, since $K(t) \geq 0$, we have
\[ \frac{du}{dt} \leq -u(t)^2 \]
and so, for $t \in (t_2, t_1)$,
\[ u(t) \leq \frac{1}{t-t_2+1/u(t_2)}. \]
Thus, the solution of the Riccati equation blows up to $- \infty$ before the time \[ t_2 - \frac{1}{u(t_2)} \leq t_2 + \frac{\sqrt{\gamma_r(\delta_2)^2 - \gamma_r(\delta_3)^2}}{1-\cos (r_{i_0}/2)}, \] thus before the time $t_1$, provided that $\delta_2$ is sufficiently small with respect to $\delta_1$.

By symmetry, the solution blows up to $+ \infty$ in negative times, after the time $-t_1$.
\end{proof}

This ends the proof of Theorem~\ref{thmImpossible}.

\section*{Acknowledgements}
This work was partially supported by the European Research Council.

\bibliographystyle{alpha}
\bibliography{ref}

\end{document}